\documentclass[12pt,reqno]{amsart}

\usepackage{amssymb}

\newtheorem{theorem}{Theorem}[section]

\theoremstyle{definition}
\newtheorem{definition}[theorem]{Definition}
\newtheorem{remark}[theorem]{Remark}

\numberwithin{equation}{section}

\begin{document}

\baselineskip=15.5pt

\title[Non-K\"ahler Calabi--Yau manifolds]{Non-K\"ahler Calabi--Yau manifolds and
holomorphic geometric structures}

\author[I. Biswas]{Indranil Biswas}

\address{Department of Mathematics, Shiv Nadar University, NH91, Tehsil Dadri,
Greater Noida, Uttar Pradesh 201314, India}

\email{indranil.biswas@snu.edu.in, indranil29@gmail.com}

\author[S. Dumitrescu]{Sorin Dumitrescu}

\address{Universit\'e C\^ote d'Azur, CNRS, LJAD, France}

\email{dumitres@unice.fr}



\date{}

\begin{abstract}
We study holomorphic geometric structures on non-K\"ahler compact complex manifolds with trivial canonical 
line bundle. For Vaisman Calabi-Yau manifolds we prove that all holomorphic geometric structures of affine 
type on them are locally homogeneous. Moreover, if the geometric structure is rigid, then the Vaisman 
manifold must be a Kodaira manifold. The proof uses a Beauville-Bogomolov type decomposition from \cite{Is} 
together with a weak form of Bochner principle for Vaisman Calabi-Yau manifolds that we prove here. Other 
results show that a compact complex manifold with self-dual holomorphic tangent bundle bearing a rigid 
holomorphic geometric structure of affine type have infinite fundamental group. We prove the same result for 
compact complex manifolds with trivial canonical line bundle having semistable holomorphic tangent bundle, 
with respect to some Gauduchon metric. We exhibit (non-K\"ahler) compact complex simply connected manifolds 
with trivial canonical line bundle that admit non-closed holomorphic one-forms.
\end{abstract}

\maketitle

\tableofcontents

\section{Introduction}

The geometry of compact K\"ahler manifolds with vanishing real first Chern class was intensively 
investigated in the last decades. A main differential geometric tool in the study of those so-called 
K\"ahler Calabi-Yau manifolds was introduced by the celebrated theorem of Yau \cite{Ya} which endows these 
manifolds with Ricci flat K\"ahler metric. The Riemannian geometry of the Ricci flat K\"ahler metric leads 
to the important Beauville-Bogomolov decomposition theorem for those manifolds which roughly speaking 
splits the manifold, up to a finite unramified cover, into a product of a flat compact complex torus and a 
simply connected K\"ahler manifold with trivial canonical line bundle \cite{Be,Bo}. Consequently, the 
canonical line bundle of a K\"ahler Calabi-Yau manifold has finite order (a finite power of it is 
holomorphically isomorphic to the trivial holomorphic line bundle). Another powerful tool which plays a 
role in the Beauville-Bogomolov decomposition theorem is the Bochner principle which asserts that any 
holomorphic tensor on a compact K\"ahler Calabi-Yau manifold is parallel with respect to a Ricci flat 
K\"ahler metric on it \cite{Be}.

Inspired by the results on geometric structures proved in \cite{Gr,DG}, and using the geometry of compact 
K\"ahler Calabi-Yau manifolds, the second-named author proved previously the following theorem:

\begin{theorem}[{\cite{Du}}]\label{thm sorin}
Let $M$ be a compact K\"ahler manifold $M$ with trivial first Chern class bearing a holomorphic geometric
structure $\phi$ of affine type. Then
\begin{enumerate}
\item $\phi$ is locally homogeneous;

\item if $\phi$ is rigid, then $M$ is covered by a compact complex torus. 
\end{enumerate}
\end{theorem}

The proof of the first statement makes use of the Bochner principle, while the second statement uses the 
Beauville--Bogomolov decomposition theorem \cite{Be,Bo}. Statement 2 of Theorem \ref{thm sorin}
generalizes a result established in \cite{IKO} for holomorphic affine connections.

Concerning non-K\"ahler manifolds, the results of Tosatti in \cite{To} triggered attention on compact 
complex manifolds having vanishing first Chern class in the Bott-Chern cohomology (see Section \ref{section 
semi-stable}) with the hope that they might inherit some geometric features of K\"ahler Calabi-Yau 
manifolds. Following Tosatti, \cite{To}, we call those manifolds non-K\"ahler Calabi-Yau manifolds. In 
particular, compact complex manifolds with trivial canonical line bundle are non-K\"ahler Calabi-Yau in 
this sense. It was proved in \cite{To} that non-K\"ahler Calabi-Yau manifolds in the Fujiki class 
$\mathcal C$ (i.e., bimeromorphic with K\"ahler manifolds) have torsion canonical line bundle (meaning
some power of the canonical line bundle is holomorphically trivial). 
Recently, a decomposition theorem and a Bochner principle were proved in \cite{BCDG} for 
non-K\"ahler Calabi-Yau manifolds in the Fujiki class $\mathcal C$ which are either Moishezon (i.e.,
birational with projective manifolds) or of dimension not more than four.

It should be mentioned that the {\it generalized Calabi-Gray manifolds} constructed by Fei in \cite{Fe} 
provide examples of simply connected non-K\"ahler Calabi-Yau manifolds with trivial canonical line bundle 
that admit non-trivial holomorphic one-forms with the property that the form vanishes at some point of the 
manifold. Consequently, the Bochner principle does not apply to those generalized Calabi-Gray manifolds 
constructed in \cite{Fe} whose holomorphic tangent bundle is not polystable with respect to any Gauduchon 
metric on it (see Section \ref{section semi-stable}).

Examples of complex structures on connected sums of $n \,\geq\, 2$ copies of $S^3 \times S^3$ giving 
non-K\"ahler Calabi-Yau threefolds with trivial canonical line bundle were constructed by Lu and Tian and 
by Friedman in \cite{Fr,LT1,LT2}. Bozhkov proved that all those examples of simply connected non-K\"ahler 
Calabi-Yau threefolds admit a stable holomorphic tangent bundle with respect to any Gauduchon metric 
\cite{Boz}. Other examples of simply connected non-K\"ahler Calabi-Yau threefolds with trivial canonical 
line bundle and of algebraic dimension zero were constructed in \cite{FP}. Notice also that examples of 
compact complex manifolds with vanishing first Chern class and holomorphic canonical bundle which is not of finite order are given in \cite{Ro}.

Recently a decomposition theorem of Beauville-Bogomolov type was proved by Istrati, \cite{Is}, for 
non-K\"ahler Calabi-Yau manifolds which are Vaisman (see \cite{Va1,Va2} for Vaisman manifolds). More 
precisely, it is proved in \cite[Theorem D]{Is} that, up to finite unramified covers, Vaisman Calabi-Yau manifolds have the structure of a principal 
elliptic orbibundle over an orbifold basis which is the product of an abelian variety with a simply 
connected projective orbifold with trivial canonical line bundle. This implies, in particular, that the canonical line 
bundle is of finite order. Her proof uses the orbifold version of the Beauville-Bogomolov decomposition for 
projective {\it pure} orbifolds (i.e., the singularities are of quotient type) proved by Campana in \cite{Ca}.

In this article we prove a weak form of the Bochner principle for Vaisman Calabi-Yau manifolds (Theorem 
\ref{Bochner}) which says the following: {\it No nontrivial holomorphic tensor on a compact Vaisman 
Calabi-Yau manifold vanishes at some point}. Using this Theorem \ref{Bochner} and the decomposition theorem 
of Istrati in \cite{Is}, we deduce the following analogous of Theorem \ref{thm sorin} for compact Vaisman 
Calabi-Yau manifolds.

\begin{theorem} \label{Vaisman}
Let $M$ be a compact Vaisman Calabi-Yau manifold bearing a holomorphic geometric structure $\phi$
of affine type. Then the following two statements hold:
\begin{enumerate}
\item $\phi$ is locally homogeneous and invariant with respect to the canonical elliptic fibration.

\item If $\phi$ is rigid, then $M$ admits a finite unramified cover which
is a Kodaira manifold (i.e., an elliptic pure orbibundle over an abelian variety).
\end{enumerate}
\end{theorem} 

Notice that all Kodaira (primary and secondary) surfaces are known to admit Vaisman metrics \cite{Bel}: 
they are Vaisman Calabi-Yau complex surfaces. On the other hand, all Kodaira primary surfaces admit complex 
affine structures and other non-flat locally homogeneous holomorphic affine connections \cite{IKO, Vi}. 
Kodaira manifolds are higher dimensional generalizations of Kodaira surfaces: they have the structure of a 
nilmanifold with a left invariant complex structure \cite{GMPN}. They are examples of Vaisman Calabi-Yau 
manifolds which are principal elliptic pure orbibundles over an abelian variety \cite[Examples 7.2]{Is}.

A by-product of the proof of Theorem \ref{Vaisman} is the generalization to the case of a principal compact 
complex torus orbi-bundle with trivial canonical bundle over a orbifold K\"ahler Ricci flat manifold of a 
result in \cite{BD2} (see Theorem \ref{orbibundle}): all holomorphic geometric structures of affine type on 
them are locally homogeneous and invariant by the principal fibration.

In contrast, minimal complex surfaces, which are principal elliptic bundles over compact Riemann surfaces of 
genus at least two, and having an odd first Betti number are known to admit non-locally homogeneous 
holomorphic affine connections \cite{Du1}. They also admit nontrivial holomorphic one-forms with non-empty 
vanishing locus that are pulled back from the base. It was proved in \cite{Bel} that all those minimal 
surfaces with Kodaira dimension one and odd first Betti number actually admit Vaisman metrics. Therefore, 
Theorem \ref{Vaisman} and the weak Bochner principle (Theorem \ref{Bochner}) do not hold for all Vaisman 
manifolds. Other examples of Vaisman manifolds which admit non-locally homogeneous rigid holomorphic 
geometric structures of affine type and nontrivial holomorphic tensors with non-empty vanishing locus are 
the diagonal Hopf manifolds (see the construction of those holomorphic geometric structures in Remark 
\ref{diagonal Hopf} and recall that diagonal Hopf manifolds are also known to admit Vaisman metrics 
\cite{Bel}).

Also in this article we extend the result proved in \cite[Theorem D]{BCDG} on rigid holomorphic 
geometric structures on simply connected manifolds with trivial canonical line bundle and prove the 
following:

\begin{theorem}\label{main} Let $M$ be a compact complex manifold $M$ with trivial canonical line bundle 
that admits a rigid holomorphic geometric structure of affine type. Then the following two statements hold: 
\begin{enumerate}

\item Assume that $TM$ is semistable with respect to some Gauduchon metric on $M$. Then the fundamental 
group of $M$ is infinite.

\item Assume that $\dim H^0(M,\, T^*M) \ =\ \dim H^0(M,\, TM)$. Then the fundamental group of $M$ is infinite.
\end{enumerate}
\end{theorem}

An analogous result was proved in \cite[Theorem D]{BCDG} in several other cases: for threefolds, for 
manifolds in the Fujiki class $\mathcal C$ and for manifolds having algebraic dimension at most one. All 
those results point in the direction of the conjecture formulated in \cite{BCDG} which asserts that simply 
connected compact complex manifolds with trivial canonical line bundle do not admit holomorphic rigid 
geometric structures of affine type.

Theorem \ref{Vaisman} and Theorem \ref{main} also hold for some important geometric structures which are 
not of affine type (see Definition \ref{def}), namely for holomorphic projective connections and for 
holomorphic conformal structures. This is because on manifolds with trivial canonical line bundle, these 
two geometric structures lift to global representatives which are of affine type, namely, a holomorphic 
affine connection and a holomorphic Riemannian metric respectively. The particular case of holomorphic 
Riemannian metrics was settled in \cite{BD3}. Statement (2) of Theorem \ref{main} generalizes this result 
to any self-dual holomorphic tangent bundle (i.e., $TM$ is isomorphic to its dual $T^*M$).

The statement (1) in Theorem \ref{main} is proved in Section \ref{sest1} and the statement (2) is proved in 
Section \ref{sest2}. Theorem \ref{Vaisman} together with the weak Bochner principle for Vaisman Calabi-Yau 
manifolds (Theorem \ref{Bochner}) are proved in Section \ref{seva}. Theorem \ref{simply connected with 
forms} in Section \ref{forms} constructs simply connected elliptic principal bundles over a $K3$-surface 
that admit non-closed holomorphic one-forms.

\section{Holomorphic rigid geometric structures}

This section introduces the notion of a holomorphic geometric structure. We follow the presentation of rigid 
geometric structures given in \cite{Gr} (see also \cite{Ben,DG}).

Let $M$ be a connected complex manifold of complex dimension $n$. Consider, for any positive integer $k 
\,\geq \,1$, the associated principal bundle of $k$-th order frames $R^k(M) \,\longrightarrow\, M$: it is 
the holomorphic principal bundle formed by the $k$-jets of local holomorphic coordinates on the manifold 
$M$. The structure group of $R^k(M)$ is the group $D^k ({\mathbb C}^n)$ of $k$-jets of local 
biholomorphisms of $\mathbb{C}^n$ fixing the origin. We denote it simply by $D^k$ and we recall that $D^k$ 
is actually a complex affine algebraic group.

\begin{definition}\label{def}
A {\it holomorphic geometric structure} of order $k$ on the complex manifold $M$ is a holomorphic
$D^k$-equivariant map $\phi\,:\, R^k(M)\, \longrightarrow\, Z$, with $Z$ being a complex
algebraic manifold endowed with an algebraic $D^k$-action. Moreover, the geometric
structure $\phi$ is said to be of {\it affine type} if $Z$ is a complex affine variety.
\end{definition}

All holomorphic tensors on $M$ (meaning holomorphic sections of a bundle $(T^*M)^{\otimes p}\otimes 
(TM)^{\otimes q}$, for some nonnegative integers $p,\,q$) are holomorphic geometric structures of affine 
type of order one. Indeed, any holomorphic tensor on $M$ can be seen as a holomorphic ${\rm GL}(n, \mathbb 
C)$-equivariant map from the first order frame bundle $R^1(M)$ to a linear complex algebraic representation 
of ${\rm GL}(n, \mathbb C)$, namely $(({\mathbb C}^n)^*)^{\otimes p}\otimes {\mathbb C^n}^{\otimes q}$, 
where ${\mathbb C}^n$ is the standard representation of ${\rm GL}(n, \mathbb C)$. Holomorphic affine 
connections are holomorphic geometric structures of affine type of order two \cite{Gr,DG,Ben}. Holomorphic 
foliations, holomorphic projective connections and holomorphic conformal structures are holomorphic 
geometric structures of non-affine type.

There is the following natural notion of symmetry of a holomorphic geometric structure.

A (local) biholomorphism of $M$ preserves a holomorphic geometric structure $\phi$ if its canonical lift to 
$R^k(M)$ fixes each fiber of the map $\phi$. Such a (local) biholomorphism is called a {\it (local) 
automorphism} of $\phi$. The group of biholomorphisms of $M$ preserving the geometric structure $\phi$ is 
called the automorphism group of $(M,\, \phi)$ and will be denoted by ${\rm Aut}(M, \,\phi)$.

A (local) holomorphic vector field on $M$ is called a (local) {\it Killing vector field} with 
respect to $\phi$ if its (local) flow acts on $M$ by (local) biholomorphisms preserving $\phi$ (i.e., it
acts by (local) automorphisms of $\phi$).

If the sheaf of local Killing vector fields with respect to $\phi$ is transitive on $M$, the holomorphic 
geometric structure $\phi$ is called {\it locally homogeneous}.

\begin{definition}\label{def-rigid}
A holomorphic geometric structure $\phi$ is called {\it rigid} of order $l$ in 
Gromov's sense if any local biholomorphism preserving $\phi$ is uniquely determined by its 
$l$-jet at any given point.
\end{definition}

Holomorphic affine connections are rigid of order one in Gromov's sense (see \cite{Gr,DG, Ben}). The 
rigidity comes from the fact that local biholomorphisms fixing a point and preserving a connection 
linearize in exponential coordinates, so they are indeed completely determined by their differential at the 
fixed point.

Holomorphic Riemannian metrics, holomorphic projective connections and holomorphic 
conformal structures in dimension $\geq 3$ are all rigid holomorphic geometric structures,
while holomorphic symplectic structures and holomorphic foliations are non-rigid geometric 
structures~\cite{DG}.

The sheaf of local Killing vector fields of a holomorphic rigid geometric structure $\phi$ is 
locally constant. Its fiber is a finite dimensional Lie algebra which is called {\it the 
Killing algebra} of $\phi$ \cite{DG,Gr}.

An extendibility result, proved first by Nomizu for analytic Riemannian metrics, \cite{No}, and generalized 
in \cite{Am, Gr,DG} for analytic rigid geometric structures, asserts that every point $m$ in $M$ has an 
open neighborhood $V_m$ in $M$ such that for any connected open subset $W\,\subset\, V_m$, any local 
Killing vector field defined over $W$ uniquely extends to a local Killing vector field defined on $V_m$. 
Therefore, by monodromy principle, all local Killing vector fields extend along any path in the manifold, 
and the extension depends only on the homotopy class of the path. In particular, for simply connected 
manifolds, local Killing vector fields extend to the entire manifold as global (uni-valued) Killing vector 
fields.

Assume now that $M$ is compact and simply connected, The global holomorphic vector fields are complete and 
they define an action of a complex Lie group. We then get that the local holomorphic Killing vector fields 
for holomorphic rigid geometric structures globalize and define an action of a connected complex Lie 
subgroup lying in ${\rm Aut}(M,\, \phi)$ and acting by automorphisms on $M$ while preserving the rigid 
geometric structure. The corresponding Lie algebra is the vector space of globally defined holomorphic 
Killing vector fields on $M$.

In this context the following result was proved in \cite[Theorem 5.1]{BCDG}.

\begin{theorem}[\cite{BCDG}]
\label{main lemma}
Let $M$ be a compact simply connected complex manifold with trivial canonical line bundle $K_M$. Assume that
$M$ bears a holomorphic rigid geometric structure $\phi$ of affine type. Then the following three
statements hold:
\begin{enumerate}
\item[(i)]\, There exists a holomorphic submersion $\pi \,:\, M\,\longrightarrow\, B$ to a simply connected 
Moishezon manifold $B$ with globally generated canonical line bundle $K_B$ such that the fibers of $\pi$ are 
complex tori.

\item[(ii)]\, The above fibration $\pi$ is not isotrivial. Equivalently, $K_B$ is not trivial.

\item[(iii)]\, There exists a maximal connected abelian subgroup $A$ of the automorphism group 
$\rm{Aut}(M,\, \phi)$ whose orbits coincide with the fibers of $\pi$. Moreover, $A$ is noncompact and its 
(real) maximal compact subgroup $K$ acts freely and transitively on the fibers of $\pi$ (hence $M$ is a 
$C^\infty$ principal $K$--bundle over $B$).
\end{enumerate}
\end{theorem}

We will use this result to prove Theorem \ref{main}. Statement (1) in Theorem \ref{main} is proved in 
Section \ref{section semi-stable} and statement (2) is proved in in Section \ref{section self-dual}.

\section{Semistable holomorphic tangent bundle}\label{section semi-stable}

Let us first recall the notion of stability in the framework of compact complex (non necessarily K\"ahler) 
manifolds.

Consider a compact complex connected manifold $M$ of complex dimension $n$. A Hermitian metric on the 
holomorphic tangent bundle $TM$ is called a \text{Gauduchon metric} if the corresponding real $(1,\, 
1)$--form $\omega$ on $M$ satisfies the equation
$$
\partial\overline{\partial}\omega^{n-1}\,=\, 0.
$$

Gauduchon proved in \cite{Ga} that any compact complex manifold admits a Gauduchon metric \cite{Ga}. In the 
sequel we fix such a Gauduchon metric on $M$, and denote by $\omega_0$ the $(1, \,1)$--form on 
$M$ associated to the Gauduchon metric.

Let $\Omega^{1,1}_{cl}(M,\, {\mathbb R})$ denote the space of all globally
defined $d$--closed real $(1,\, 1)$--forms on $M$. The Bott--Chern cohomology of $M$ is defined as 
$$
H^{1,1}_{BC}(M,\, {\mathbb R})\ :=\ \frac{\Omega^{1,1}_{cl}(M,\, {\mathbb R})}{
\{\sqrt{-1}\partial\overline{\partial}\alpha\,\mid\, \alpha\,\in\,
C^{\infty}(M,\,{\mathbb R})\}} .
$$
The integration mapping 
$$
\Omega^{1,1}_{cl}(M,\, {\mathbb R})\ \longrightarrow\ {\mathbb R},\ \ \
\alpha\ \longmapsto\ \int_M \alpha\wedge \omega^{n-1}_0
$$
descends to a linear form on the Bott-Chern cohomology $H^{1,1}_{BC}(M,\, {\mathbb R})$.

For any holomorphic line bundle $L$ on $M$, choose a Hermitian metric $h_L$ on $L$, and consider the 
associated element $c(L)\,\in\, H^{1,1}_{BC}(M,\, {\mathbb R})$ given by the curvature of the Chern 
connection on $L$ associated to $h_L$. This cohomology class $c(L)$ does not depend on the choice of the 
Hermitian metric $h_L$. Clearly, $c(L)$ vanishes for the trivial holomorphic bundle.

For any torsionfree coherent analytic sheaf $F$ on the manifold $M$, define
$$
\text{degree}(F)\,:=\, \int_M c(\det F)\wedge \omega^{n-1}_0\, \in\, \mathbb R\, ,
$$
where $\det F$ is the determinant line bundle for
$F$ \cite[p.~166, Proposition~6.10]{Ko}. The real number
$\text{degree}(F)/\text{rank}(F)$ is called the \textit{slope} of $F$ and it is denoted
by $\mu(F)$. 

A torsionfree coherent analytic sheaf $V$ on $M$ is called
\textit{stable} if
$$
\mu(F)\, <\, \mu(V)
$$
for all coherent analytic subsheaf $F\, \subset\, V$ with $0\, <\, \text{rank}(F)\,
<\, \text{rank}(V)$. A torsionfree coherent analytic sheaf $V$ on $M$ is called
\textit{semistable} if
$$
\mu(F)\, \leq \, \mu(V)
$$
for all coherent analytic subsheaf $F\, \subset\, V$ with $0\, <\, \text{rank}(F)\,
<\, \text{rank}(V)$.

If $V$ is a direct sum of stable sheaves of the same slope, then it is
called \textit{polystable}. 

The reader is referred to \cite{LT} for more details and results about the stability.

Let us now give the proof of statement (i) in Theorem \ref{main}.

\subsection{Proof of statement (i) in Theorem \ref{main}}\label{sest1}

Let $M$ be a simply connected compact complex manifold with trivial canonical line bundle and $\Psi$ a rigid 
holomorphic geometric structure of affine type on $M$.

Let
\begin{equation}\label{e1}
\pi\ :\ M\ \longrightarrow\ B
\end{equation}
be the holomorphic submersion given by Theorem \ref{main lemma}, and denote by $A$ the connected abelian 
Lie subgroup of ${\rm Aut}(M, \Psi)$ in Theorem \ref{main lemma}(iii) whose orbits coincide with the fibers 
of $\pi$.

Fix a Gauduchon metric $\omega$ on $M$. Assume that the holomorphic tangent bundle $TM$ of $M$ is 
semistable with respect to $\omega$.

Denote by $\mathfrak a$ the Lie algebra of $A$. The trivial holomorphic vector bundle on $M$
$$
M\times {\mathfrak a}\ \longrightarrow\ M
$$
with fiber $\mathfrak a$ will be denoted by $\mathcal A$. The differential of the action of $A$ on $M$
produces an ${\mathcal O}_M$--linear holomorphic homomorphism 
\begin{equation}\label{e2}
\Phi\ :\ \mathcal A \ \longrightarrow\ TM.
\end{equation}

Notice that $\Phi(\mathcal A)$ coincides with $T_{\pi}$, with $T_\pi\, \subset\, TM$ is the vertical
tangent bundle for the projection $\pi$ in \eqref{e1}. 

Consider the kernel
\begin{equation}\label{e3}
{\mathcal K}\ := \ \text{kernel}(\Phi) \ \subset\ \mathcal A
\end{equation}
of the homomorphism $\Phi$ in \eqref{e2}. We have $\text{degree}(TM)\,=\, 0$ because the holomorphic
line bundle $\det(TM)\,=\, \bigwedge^{\rm top}TM$ is trivial (recall the assumption that the
canonical line bundle of $M$ is trivial). Since the holomorphic tangent bundle $TM$ is semistable of
degree zero, and $\Phi(\mathcal A)\, \subset\, TM$, it follows that
\begin{equation}\label{e4}
\text{degree}(\Phi(\mathcal A)) \ \leq \ 0.
\end{equation}
On the on other hand, $\Phi(\mathcal A)$ is a quotient of the holomorphically trivial vector
bundle $\mathcal A$, which implies that
\begin{equation}\label{e5}
\text{degree}(\Phi(\mathcal A)) \ \geq \ 0,
\end{equation}
because the trivial vector bundle is semistable.
Combining \eqref{e4} and \eqref{e5} we conclude that $\text{degree}(\Phi(\mathcal A)) \ = \ 0$.
Since $$\text{degree}({\mathcal A})\ =\ 0\ =\ \text{degree}(\Phi(\mathcal A)),$$ it now follows that
\begin{equation}\label{dk}
\text{degree}({\mathcal K}) \ = \ 0
\end{equation}
(see \eqref{e3}).

Note that $\Phi(\mathcal A)$ is torsionfree because it is a subsheaf of the torsionfree sheaf $TM$. This 
and the fact that $\mathcal A$ is locally free together imply that $\mathcal K$ is reflexive \cite[Ch.~V, 
p.~153, Proposition 4.13]{Ko}. Since the trivial vector bundle $\mathcal A$ is polystable of degree zero, 
and $\mathcal K$ is a reflexive subsheaf of it of degree zero (see \eqref{dk}), it follows that $\mathcal 
K$ is a polystable subbundle of $\mathcal A$ of degree zero. All stable subbundles of $\mathcal A$ of 
degree zero are of the form $M\times \ell$, where $\ell\, \subset\, {\mathfrak a}$ is a line. Consequently, 
there is a linear subspace $V_0\, \subset\, {\mathfrak a}$ such that $\dim V_0\ =\ {\rm rank}({\mathcal 
K})$ and
$$
{\mathcal K}\ =\ M\times V_0\ \subset\ M\times {\mathfrak a}\ =\ {\mathcal A}.
$$

Take any $v\, \in\, V_0\, \subset\, {\mathfrak a}$. Note that the vector field $\Phi(v)\,\in\, H^0(M,\, 
TM)$ (see \eqref{e2}) vanishes identically because $\Phi(v)\,\in\, H^0(M,\, {\mathcal K})$. This implies 
that $v\,=\, 0$. Therefore, we conclude that $\Phi$ defines an isomorphism
\begin{equation}\label{e6}
{\mathcal A}\ \simeq \ T_\pi \ \subset\ TM,
\end{equation}
where $T_\pi$ as before is the vertical tangent bundle for the projection $\pi$ in \eqref{e1}. From
\eqref{e6} it follows that $T_\pi$ is a trivial bundle, and consequently,
$$\det T_\pi\ =\ {\mathcal O}_M .$$
Hence we have
\begin{equation}\label{e7}
\det (\pi^*TB)\ =\ \pi^*\det (TB) \ =\ \det (TM)\otimes (\det T_\pi)^* \ =\ {\mathcal O}_M.
\end{equation}
Since the fibers of $\pi$ are compacts and  connected, from \eqref{e7} we will deduce below 
that $\det (TB)\,=\, {\mathcal O}_B$.  We will show that for a holomorphic line
bundle $\mathcal L$ on
$B$ with the property that the pullback $\pi^* {\mathcal L}$ on $M$ is holomorphically trivial, the following
holds:
\begin{equation}\label{rs}
{\mathcal L} \ =\ {\mathcal O}_B.
\end{equation}
Since the fibers of $\pi$ are  connected, we  have  $\pi_*{\mathcal O}_M\,=\, {\mathcal O}_B$; on each
reduced fiber of $\pi$, the space of global functions (defined on the fiber) are constants because the fiber
is compact and connected. Consequently, the projection formula gives that
$$
\pi_*\pi^*{\mathcal L}\ =\ {\mathcal L}\otimes \pi_*{\mathcal O}_M\ =\ {\mathcal L}.
$$

Fix a holomorphic isomorphism
between $\pi^*\mathcal L$ and ${\mathcal O}_M$.

On the other hand, we have
\begin{equation}\label{rs2}
{\mathbb C}\,=\, H^0(M,\, {\mathcal O}_M) \, =\, H^0(M, \, \pi^* {\mathcal L})\,=\, H^0(B,\,
\pi_*\pi^*{\mathcal L})\, =\, H^0(B,\, {\mathcal L}).
\end{equation}
Using \eqref{rs2}, a nonzero section of $\pi^* {\mathcal L}\,=\, {\mathcal O}_M$ gives
a section of $\mathcal L$ which trivializes $\mathcal L$. This proves \eqref{rs}

Thus we have $\det (TB)\,=\, {\mathcal O}_B$. Note that
this is in contradiction with the statement (ii) in Theorem \ref{main lemma} and hence
the proof in the case where $M$ is simply connected is complete.

If the fundamental group of $M$ is finite, then there exists a finite unramified covering of $M$ which is 
simply connected and endowed with the pull-back of $\Phi$ which is a rigid holomorphic geometric structure 
of affine type. This leads to a contradiction as seen above.

\section{Self-dual holomorphic tangent bundle} \label{section self-dual} 

Now assume that
\begin{equation}\label{e8}
\dim H^0(M,\, T^*M) \ =\ \dim H^0(M,\, TM).
\end{equation}

Notice that an interesting particular case where the condition \eqref{e8} is satisfied is when the 
holomorphic tangent bundle $TM$ is self-dual.

Let us now give the proof of point (ii) in Theorem \ref{main}.

\subsection{Proof of statement (ii) in Theorem \ref{main}}\label{sest2}

As before, assume that $M$ is simply connected with trivial canonical bundle and $\Psi$ a rigid holomorphic 
geometric structure of affine type on it. Consider the holomorphic submersion in \eqref{e1}. We use the 
same notation as in the previous Section.

We have seen that $M$ bears global Killing holomorphic vector fields: they are the fundamental vector 
fields for the holomorphic action of the abelian complex Lie subgroup $A \, \subset\, {\rm Aut}(M, \phi)$
acting transitively on the fibers of $\pi$. Therefore,
$$\dim H^0(M,\, T^*M) \ =\ \dim H^0(M,\, TM) \geq \ \dim A \ \geq \ \dim M - \dim B\ >\ 0.$$

Take any nonzero holomorphic $1$-form
\begin{equation}\label{eth}
0\ \not=\ \theta\ \in\ H^0(M,\, \Omega^1_M)\ =\ H^0(M,\, T^*M).
\end{equation}
We will show that there is some $b\, \in\, B$ such that the section of $\Omega^1_{\pi^{-1}(b)}$ obtained by 
restricting $\theta$ (see \eqref{eth}) to the compact complex torus $\pi^{-1}(b)$ is nonzero. To prove this 
by contradiction, assume that the section of $\Omega^1_{\pi^{-1}(b)}$ obtained by restricting $\theta$ to 
the compact complex torus $ \pi^{-1}(b)$ is zero for all $b\, \in\, B$. Then there is a section $\theta'\, 
\in\, H^0(B,\, \Omega^1_B)$ such that $$\theta\ =\ \pi^*\theta'.$$

Since $B$ is Moishezon, holomorphic forms on $B$ are closed. The manifold $B$ being also simply connected, 
this implies that $H^0(B,\, \Omega^1_B)\,=\,0$. Consequently, $\theta'$ and (hence also) $\theta$ vanish 
identically. This contradicts \eqref{eth}, and hence there is some $b\, \in\, B$ such that the section of
$\Omega^1_{\pi^{-1}(b)}$ obtained by restricting $\theta$ to $\pi^{-1}(b)$ is nonzero.

A stronger version of the above statement will be proved. We will now show that the 
section of $\Omega^1_{\pi^{-1}(b)}$ obtained by restricting $\theta$ is nonzero for
\textit{every} $b\, \in\, B$.

To prove this by contradiction, assume that the section of $\Omega^1_{\pi^{-1}(b)}$ obtained by restricting 
$\theta$ is identically zero for some $b\, \in\, B$. Take any $v\, \in\, {\mathfrak a}$, and consider the 
holomorphic function $\theta (v)$ on $M$. Since the section of $\Omega^1_{\pi^{-1}(b)}$ obtained by 
restricting $\theta$ is identically zero, and the vector field $v$ is vertical for the projection $\pi$, we 
know that the function $\theta (v)$ vanishes on $\pi^{-1}(b)$. This implies that the (constant) function 
$\theta (v)$ vanishes identically on $M$.

On the other hand, the image of the homomorphism $\Phi$ (see \eqref{e2}) is the entire $T_\pi$ (see 
\eqref{e6}). Therefore, the above observation that the function $\theta (v)$ vanishes identically on $M$ 
for all $v\, \in\, {\mathfrak a}$ implies that the restriction of $\theta$ to the entire $T_\pi$ vanishes 
identically. But this contradicts the earlier observation that there is some $b\, \in\, B$ such that the 
section of $\Omega^1_{\pi^{-1}(b)}$ obtained by restricting $\theta$ is nonzero. Hence we conclude that the 
section of $\Omega^1_{\pi^{-1}(b)}$ obtained by restricting $\theta$ is nonzero for every $b\, \in\, B$.

Since the section of $\Omega^1_{\pi^{-1}(b)}$ obtained by restricting $\theta$ is nonzero for every
$b\, \in\, B$, we have
\begin{equation}\label{e9}
\dim H^0(M,\, \Omega^1_M) \ \leq\ \dim H^0(\pi^{-1}(b),\, \Omega^1_{\pi^{-1}(b)})\ =\ \dim M - \dim B.
\end{equation}
On the other hand,
\begin{equation}\label{e10}
\dim H^0(M,\, \Omega^1_M) \ = \ \dim H^0(M,\, TM) \ \geq\, \dim M - \dim B
\end{equation}
(see \eqref{e8}). Combining \eqref{e9}, \eqref{e10} and \eqref{e8} we have the following:
$$
\dim H^0(M,\, \Omega^1_M) \ =\ \dim H^0(M,\, TM) \ =\ \dim M - \dim B.
$$
Hence the homomorphism $\Phi$ in \eqref{e2} is an isomorphism between $\mathcal A$ and $T_\pi$.
In particular, the holomorphic vector bundle $T_\pi$ is trivial. 

We now have (as in \eqref{e7})
$$
\det (\pi^*TB)\ =\ \pi^*\det (TB) \ =\ \det (TM)\otimes (\det T_\pi)^* \ =\ {\mathcal O}_M.
$$
Since the fibers of $\pi$ are compacts and  connected, this implies that $\det (TB)\,=\, {\mathcal O}_B$ (see
\eqref{rs} and its proof following \eqref{rs2}). This is in contradiction with statement (ii) in Theorem~\ref{main lemma}, and
hence the proof is complete.

\section{Vaisman Calabi-Yau manifolds}\label{section Vaisman} 

Vaisman manifolds were introduced in \cite{Va1}. This is a class of Hermitian non-K\"ahler manifolds which 
do not satisfy the $\partial \overline{\partial}$-lemma, but still possess a cohomological behaviour close to 
that of K\"ahler manifolds.

Recall that a Hermitian manifold $M$ is called Vaisman if its fundamental $(1,\,1)$--form $\Omega$ is not 
close, but there exists a closed one-form $\theta$, called the {\it Lee form}, satisfying the conditions 
that $d\Omega\,=\, \Omega \wedge \theta$ and $\theta$ is parallel with respect to the Levi-Civita 
connection. The vector field dual to the Lee form --- known as the {\it Lee vector field} --- is real 
holomorphic and Killing \cite{Va2}: it generates a complex one-parameter subgroup in the holomorphic 
automorphism group of the complex manifold $M$ which does not depend on the chosen Vaisman metric (see 
\cite{Ts}). This one-parameter complex group generated by the Lee vector field --- called {\it the Lee 
group} of the manifold --- acts locally freely on $M$: its action defines a holomorphic foliation --- known 
as the {\it canonical foliation} --- which is actually transversally K\"ahler \cite{Va2}. If the Lee group is 
compact, then the Vaisman manifold is called {\it quasi-regular} and in this case the quotient of $M$ by 
the action of the Lee group is in fact a complex orbifold. If, moreover, the action of the compact Lee group is 
free, then $M$ is called a {\it regular} Vaisman manifold, and the quotient of $M$ by the Lee group is 
smooth: therefore, in the regular case, $M$ has the structure of a holomorphic principal elliptic bundle 
over a complex smooth base. For more informations about Vaisman manifolds and more generally about locally conformally K\"ahler manifolds we refer the reader to the book \cite{OV}.

Following Tosatti, \cite{To}, compact complex manifolds with vanishing first Chern class in 
Bott-Chern cohomology are called Calabi-Yau.

Vaisman Calabi-Yau manifolds have been studied by Istrati in \cite{Is}. By Theorem \cite[Theorem C]{Is}, the 
connected component of the automorphism group of any Vaisman Calabi-Yau manifold is compact one dimensional and it
coincides with the Lee group. Therefore, any Vaisman Calabi-Yau manifold is quasi-regular. Moreover, a 
Beauville-Bogomolov type decomposition Theorem for Vaisman Calabi-Yau manifolds is proved in \cite[Theorem 
D]{Is}. More precisely, \cite[Theorem D]{Is} shows that any Vaisman Calabi-Yau manifold $M$ admits a finite 
unramified cover which is a principal elliptic orbi-bundle over an orbifold Y which is the product of an 
abelian variety and a simply connected projective orbifold with trivial canonical class. This implies, in 
particular, that the canonical bundle of a Vaisman Calabi-Yau manifold is of finite order.
 
Now we make use of the above results of Istrati \cite{Is} to prove Theorem \ref{Vaisman}.

\subsection {Proof of Theorem \ref{Vaisman}}\label{seva}

Assume that the Vaisman Calabi-Yau manifold $M$ bears a holomorphic geometric structure of affine type 
$\Psi$. We pull-back the holomorphic geometric structure on the finite unramified cover given by the 
decomposition theorem \cite[Theorem D]{Is}. Therefore, we assume --- without any loss of generality --- 
that $M$ has a holomorphically trivial canonical line bundle and $M$ has the structure of a principal 
elliptic orbi-bundle
\begin{equation}\label{dpi}
\pi\ :\ M \ \longrightarrow\ Y
\end{equation}
over the projective pure orbifold $Y$ with trivial canonical class.

Denote by $T$ the elliptic curve which is the structure group of the above elliptic orbi-bundle, and also 
denote by $X$ the fundamental holomorphic vector field on $M$ defined by the holomorphic $T$-action on $M$. 
The holomorphic vector field $X$ does not vanish (since the $T$-action is locally free) and it spans the 
kernel of $d \pi$ (see \eqref{dpi}), defining the vertical subbundle $V \,\subset\, TM$ for the projection 
$\pi$.

To prove by contradiction, assume that the holomorphic geometric structure of affine type $\Psi$ is not 
locally homogeneous on $M$. Then \cite[Lemma 3.2]{Du} proves that there exist two integers $a,\,b \,\geq\, 0$ (at 
least one of them being positive) and a nontrivial holomorphic section of the holomorphic vector bundle 
$(TM)^{\otimes a}\otimes (T^*M)^{\otimes b}$ which vanishes at some point in $M$.

Observe that all holomorphic sections of $(TM)^{\otimes a} \otimes (T^*M)^{\otimes b}$ are automatically 
$T$-invariant. Indeed, to any holomorphic section $\psi$ of this vector bundle and any $t \,\in\, T$ one 
associates the section ${t^*} \psi$, given by the linear $T$-action on the vector space ${\rm H}^0(M,\, 
(TM)^{\otimes a} \otimes (T^*M)^{\otimes b})$ canonically induced by the $T$-action on $M$. Therefore, we 
get a complex Lie group homomorphism $T \,\longrightarrow \, {\rm GL}({\rm H}^0(M,\, (TM)^{\otimes a} 
\otimes (T^*M)^{\otimes b})$. But any such homomorphism from a compact complex torus to a complex affine 
group must be the trivial homomorphism. This implies that $t^*\psi\,=\,\psi$, for all $t \,\in\, T$ and all 
$\psi\,\in\, {\rm H}^0(M,\, (TM)^{\otimes a}\otimes(T^*M)^{\otimes b})$. Consequently, every holomorphic 
tensor on $M$ must be $T$-invariant.

The canonical line bundle $K_M$ being trivial, there is a canonical contraction isomorphism between $TM$ 
and $\bigwedge^{n-1} (T^*M)$. Through this isomorphism, the nontrivial holomorphic section of the 
holomorphic vector bundle $(TM)^{\otimes a}\otimes (T^*M)^{\otimes b}$ --- constructed above --- provides a 
nontrivial holomorphic section $\phi$ of $(T^*M)^{\otimes m}$, with $m\,=\,(n-1)a+b \,\geq\, 1$, such that 
$\phi$ vanishes at some point $m_0 \,\in\, M$. Recall that this section $\phi$ was proved to be 
$T$-invariant.

We will get a contradiction using an induction on the above integer $m \,\geq\, 1$.

First consider the case where $m\,=\,1$. In this case $\phi$ is a nontrivial holomorphic one-form on $M$ 
vanishing at $m_0 \,\in\, M$. Note that $\phi(X)$, where $X$ as before is the fundamental vector
field on $M$ defined by the holomorphic $T$-action on $M$, is a holomorphic function on $M$ vanishing at $m_0$, and 
therefore it vanishes identically on $M$. This proves that the vertical space $V$ (for the projection $\pi$ 
in \eqref{dpi}) is in the kernel of $\phi$. The fibers of $\pi$ being compact and connected we can project 
$\phi$ on $Y$ using $d \pi$. This constructs a nontrivial holomorphic one-form $\phi'$ on $Y$ which 
vanishes at $\pi(m_0) \,\in \,Y$ and satisfies the condition that ${\pi^*}\phi' \, =\,\phi$. Notice that at 
the points in $M$ having a nontrivial (finite) stabilizer ${\rm Stab}(m)$ in $T$, the holomorphic tensor 
$\phi$ being $T$-invariant is, in particular, ${\rm Stab}(m)$--invariant. Consequently, $\phi$ descends as a 
holomorphic tensor $\phi'$ on $Y$ which is well-defined on the orbifold $Y$.

Since $Y$ is a projective Calabi-Yau orbifold, it admits a Ricci flat orbifold K\"ahler metric; see 
\cite[Th\'eor\`eme 4.1]{Ca} for an extension of Yau's Theorem, \cite{Ya}, to the orbifold case. Holomorphic 
tensors on $Y$ are parallel with respect to any Ricci flat orbifold K\"ahler metric on it (see 
\cite[Theorem 8.2]{GGK} and \cite[Theorem A]{CGGN}). Hence the tensor $\phi'$ cannot vanish at the point 
$\pi(m_0)$ without being trivial. For a detailed explanation of this, notice that in the case where
${\rm Stab}(m_0)$ 
in $T$ is a nontrivial (finite) subgroup in $T$, a neighborhood of $\pi(m_0)$ in $Y$ is isomorphic to a 
quotient of a ball $U \,\subset\, \mathbb C^n$ centered in the origin by the action of the finite group ${\rm 
Stab}(m_0)$ fixing the origin. The holomorphic tensor $\phi$ pulls back as a ${\rm Stab}(m_0)$--invariant 
holomorphic tensor on $U$ which is parallel with respect to the pull-back, on $U$, of the K\"ahler Ricci
flat metric (which is ${\rm Stab}(m_0)$--invariant as well). Therefore, $\phi$ cannot vanish at $m_0$
without being identically zero.

So we get a contradiction when $m\,=\,1$.

Now consider the case where $m\,>\,1$. The induction hypothesis is that a nontrivial holomorphic section of 
$(T^*M)^{\otimes k}$ does not vanish if $0\,<\,k\,<\,m$.

It will be shown that for all $0 \,<\, r \,\leq\, m$, any contraction of $\phi$ with $r$-copies of the
vector field $X$ vanishes identically on $M$.

First consider the case of $r\,=\,m$. When contracted with $m$-copies of $X$, the tensor $\phi$ produces 
the holomorphic function $\phi(X,\,X,\, \cdots,\,X)$ on $M$ which vanishes at $m_0$. Therefore the function 
$\phi(X,\,X,\, \cdots,\,X)$ vanishes identically on $M$.

Now assume that $0\,<\,r\,<\,m$. When contracted with any $r$-copies of $X$, the tensor $\phi$
produces a holomorphic section of $(T^*M)^{\otimes m-r}$ vanishing at the point $m_0$. Since $0\, < m-r\, <\,
m$, the induction hypothesis holds for $m-r$. Consequently, this section of $(T^*M)^{\otimes m-r}$ vanishes
identically on $M$.

Since all the above contractions vanish identically, it follows that our tensor $\phi$ is a holomorphic 
section of $((TM/V)^*)^{\otimes m}$. Moreover, we proved the $T$-invariance of $\phi$.

The fibers of $\pi$ being compact and connected, we can project the previous tensor on $Y$ using the 
differential $d \pi$. This produces a nontrivial holomorphic section of $((TY)^*)^{\otimes m}$ which 
vanishes at the point $\pi(m_0) \,\in\, Y$. But, as before, this holomorphic tensor should be parallel with 
respect to any K\"ahler Ricci flat metric on the projective orbifold $Y$ (see \cite[Theorem 8.2]{GGK} and 
\cite[Theorem A]{CGGN}): a contradiction. This completes the proof of the local homogeneity property of the 
holomorphic geometric structure $\Psi$ of affine type.

Let us prove that the geometric structure $\Psi$ is $T$-invariant. The proof is similar to the one 
already given for tensors.

Assume that $\Psi$ is defined as a holomorphic $D^k$-equivariant map $\phi\, :\, R^k(M) \,\longrightarrow\, 
Z$, with $Z$ being a complex affine variety endowed with an algebraic $D^k$-action (see Definition 
\ref{def}).

A classical result of algebraic group representations furnishes a $D^k$--equivariant algebraic embedding 
of $Z$ into a complex vector space $W$ endowed with a $D^k$--algebraic linear action (see, for instance, 
\cite{PV}). This identifies $\Psi$ with a holomorphic section of the holomorphic vector bundle $W_M$ over 
$M$ with fiber type $W$ associated to the principal bundle $R^k(M)$ via the $D^k$--algebraic linear action 
on $W$. As in the above proof for holomorphic tensors (corresponding to the case $k\,=\,1$), observe that 
any complex Lie group homomorphism $T\,\longrightarrow\, {\rm GL} (H^0(M,\, W_M))$ is trivial, and 
therefore ${t^*}(\Psi)\,=\,\Psi$, for any $t \,\in\, T$ (here the $T$--action on $H^0(M,\, W_M)$ is 
inherited by the canonical lift of the $T$--action on the $k$--frame bundle $R^k(M)$). Consequently, $T$ 
acts via automorphisms for the holomorphic geometric structure $\Psi$. In other words, $\Psi$ is 
$T$--invariant. This completes the proof of statement (1) in Theorem \ref{Vaisman}.

In order to prove the statement (2) in Theorem \ref{Vaisman}, consider the sheaf of local Killing vector fields 
for $\Psi$. Since $\Psi$ is locally homogeneous this sheaf is transitive on $M$. We proved that the fundamental
vector field $X$ is a Killing vector field and therefore it is a global section of the sheaf of local
Killing vector fields.

Since $X$ is a global holomorphic vector field defined on $M$, the holomorphic geometric structure 
$(\Psi,\, X)$ obtained by the juxtaposition of $\Psi$ and $X$ is also a holomorphic geometric structure on 
$M$ \cite{Gr,DG, Ben}. Moreover, if $\Psi$ is rigid, then $(\Psi,\, X)$ is also rigid.

The holomorphic geometric structure $(\Psi,\,X)$ is locally homogeneous by statement (i) in Theorem 
\ref{Vaisman} applied to $(\Psi, X)$. This means that the sheaf of local Killing vector fields for $(\Psi, 
\,X)$ is transitive on $M$ and $X$ is a global section of it. Observe that this sheaf is the centralizer of 
$X$ in the sheaf of local Killing vector fields for $\Psi$ (in other words, the sheaf of local Killing 
vector fields for $\Psi$ that commute with $X$). Denote this sheaf of vector fields by ${\rm Kill}_X$ and 
consider its direct image $\nu$ through the fibration $\pi\,:\, M\,\longrightarrow\, Y$. Since the local 
sections of ${\rm Kill}_X$ commute with $X$, the sheaf ${\rm Kill}_X$ is actually $T$-invariant. Therefore, 
for any open set $U \subset M$ and any local section $Z \in {\rm Kill}_X(U)$ the pointwise projection 
$d\pi(Z)$ through the fibration $\pi$ defines a holomorphic section over $\pi(U)$ of the holomorphic tangent 
bundle $TY$. Consequently, the image of the pointwise projection through $\pi$ of the sheaf ${\rm Kill}_X$ 
defines a subsheaf $\nu$ of the holomorphic tangent sheaf $TY$. Moreover, since ${\rm Kill}_X$ is transitive 
on $M$, the subsheaf $\nu$ of the holomorphic tangent sheaf $TY$ is transitive on $Y$.

Let us now prove --- under the assumption that $\Psi$ is rigid --- that the local sections of $\nu$ 
uniquely extend to global sections of $\nu$ over all simply connected open subsets of $Y$. For this 
purpose, we prove the following extension property for the subsheaf $\nu$: Every point $y$ in $Y$ has an open 
neighborhood $V_y$ in $Y$ such that, for any connected open subset $W \,\subset\, V_y$, any local section of 
$\nu$ over $W$ uniquely extends to a local section of $\nu$ over $V_y$.

First consider a regular point $y \,\in\, Y$. Chose a simply connected open subset $V_y\, \subset\, Y$ such 
that $y \,\in\, V_y$ and $\pi^{-1}(V_y)\,=\, V_y \times T$. The fundamental group of $\pi^{-1}(V_y)$ is 
generated by the fundamental group of the fiber $T$. More precisely, if the elliptic curve $T$ is 
uniformized as a quotient of $\mathbb C$ by a normal lattice $\Gamma$ in $\mathbb C$, then the universal 
cover of $\pi^{-1}(V_y)$ is biholomorphic to $V_y \times \mathbb C$ and $\pi^{-1}(V_y)\,=\, V_y \times T$ is 
the quotient of $V_y \times \mathbb C$ by the fundamental group identified with $\Gamma$ acting trivially on 
the first factor $V_y$ and by translations on the second factor $\mathbb C$. By the extendibility result of 
local Killing vector fields for the holomorphic rigid geometric structure $(\Psi,\,X)$ (see \cite{Am, 
Gr,DG,No}), the pull-back of the sheaf ${\rm Kill}_X$ on the universal cover $V_y \times \mathbb C$ of 
$\pi^{-1}(V_y)$ is generated by its global sections, i.e., any local section of the sheaf actually uniquely 
extends to a global holomorphic vector field on $V_y \times \mathbb C$ whose flow preserves the pull-back of 
$(\Psi,\, X)$. Moreover, these global sections commute with the action of $\mathbb C$ by translations on the 
fibers (generated by the pull-back of $X$). In particular, the global sections of the pull-back of ${\rm 
Kill}_X$ on $V_y \times \mathbb C$ are $\Gamma$--invariant; consequently, they descend as global sections of 
${\rm Kill}_X$ over $\pi^{-1}(V_y)\,=\, V_y \times T$. In other words, the local sections of the sheaf ${\rm 
Kill}_X$ restricted to $\pi^{-1}(W)$ uniquely extend as holomorphic Killing vector fields on entire 
$\pi^{-1}(V_y)$. Consequently, the restriction of the above defined subsheaf $\nu$ of the tangent sheaf $TY$ 
to $V_y$ (recall that $\nu$ is the image of the pointwise projection through $\pi$ of the sheaf ${\rm 
Kill}_X$) shares the property that all its local sections defined on the connected open subset $W \subset 
V_y$ uniquely extend to local sections defined on $V_y$.

Next consider a point $y\,\in\, Y$ such that the points in the fiber $\pi^{-1} (y)$ have a nontrivial 
(finite) stabilizer ${\rm Stab(y)}\, \subset\, T$ (since $T$ is abelian, points of $\pi^{-1} (y)$ have a 
common stabilizer). In this case $y$ admits a simply connected open neighborhood $V_y$ in $Y$ such that 
$\pi^{-1}(V_y)= (V_y \times T)/{\rm Stab(y)},$ with the cyclic group ${\rm Stab(y)}$ acting by rational 
rotations on $V_y$ (identified with a ball) and it acts by translations on the fiber $T$. As before, the 
sheaf obtained by the pull-back of ${\rm Kill}_X$ (restricted to $\pi^{-1}(V_y)$) to $V_y \times T$ is 
generated by its global sections. Since those global sections are $T$--invariant (they commute with the 
pull-back of $X$), they are in particular ${\rm Stab(y)}$--invariant. Therefore, the sheaf ${\rm Kill}_X$ 
restricted to $\pi^{-1}(V_y)\,=\, (V_y \times T)/{\rm Stab(y)}$ is generated by its global sections. We 
conclude as before that its pointwise projection $\nu$ through $\pi$ is generated by its global sections on 
$V_y$.

This implies that the sheaf $\nu$ on $Y$ has the extendibility property: Its sections extend by analytic 
continuation along any path in $Y$. Moreover, by the monodromy principle, this extension only depends on 
the homotopy type (with fixed endpoints) of the chosen path. In particular, any local section of $\nu$ 
defined on an open subset $V_y$ in $Y$ extends as a global holomorphic vector field on any connected simply 
connected open subset containing $V_y$.

Recall now that $Y$ is the product of an abelian variety with a simply connected projective pure orbifold 
with trivial canonical class. Assume, by contradiction, that the simply connected projective pure orbifold 
factor has positive dimension and denote this orbifold by $Y_1$. Consider a simply connected open set $U$ of the 
abelian variety, and consider the restriction $\nu_1$ of the subsheaf $\nu$ to the simply connected open subset 
$U \times Y_1 \,\subset\, Y$. Recall that $\nu_1$ is transitive on $U \times Y_1$ and that the extendability 
property of $\nu_1$ implies that $\nu_1$ is spanned by its global sections.

Considering the canonical projection $T(U \times Y_1)={\pi_1}^{*}(T_U) \oplus {\pi_2}^{*}(TY_1) 
\longrightarrow {\pi_2}^{*} TY_1$ restricted to $\{u \} \times Y_1$ for some $u \in U$ (with $\pi_1$ 
(respectively, $\pi_2$) being the projection on to the first (respectively, second) factors), this implies 
that the holomorphic tangent bundle of the submanifold $\{u \} \times Y_1$ is globally generated. But we 
have seen that, by Bochner principle, all holomorphic vector fields on $Y_1$ are invariant by the parallel 
transport of a Ricci flat orbifold K\"ahler metric on it (see \cite[Theorem 8.2]{GGK} and \cite[Theorem 
A]{CGGN}). Hence $TY_1$ admits a flat holomorphic trivialization which implies $Y_1$ is an abelian variety: 
a contradiction.

This completes the proof of Theorem \ref{Vaisman}.

Notice that in the proof of the statement (1) of Theorem \ref{Vaisman} the following weak 
version of Bochner principle is proved:

\begin{theorem}\label{Bochner}
Let $M$ be a compact Vaisman Calabi-Yau manifold. Then any nontrivial holomorphic tensor on $M$ has empty 
vanishing locus.
\end{theorem}

Another by-product of the proof of the statement (1) of Theorem \ref{Vaisman} is the following.

\begin{theorem}\label{orbibundle}
Let $M$ be a compact complex manifold, with trivial canonical line bundle, which is a principal torus
orbibundle over a K\"ahler pure orbifold with trivial canonical class. Then any holomorphic geometric structure 
of affine type on $M$ is locally homogeneous and invariant by the principal fibration.
\end{theorem}

Indeed, the proof of Theorem \ref{Vaisman}(1) directly generalizes from the case of a holomorphic principal 
$T$--orbi-bundle, with $T$ an elliptic curve, to the case where $T$ is a compact complex torus. The above 
statement extends to orbi-bundles a result obtained earlier in \cite[Theorem 1.2]{BD2}.

\begin{remark}\label{diagonal Hopf}
Diagonal Hopf manifolds $M_D$ are defined as quotients of ${\mathbb C}^n \setminus \{0 \}$ by a diagonal 
linear contraction matrix $D$. Hopf manifolds possess a natural complex affine structure 
(given by a torsionfree flat affine connection $\nabla_0$) which is constructed
as the descend of the standard affine 
structure of the universal cover ${\mathbb C}^n \setminus \{0 \}$. Moreover the commuting holomorphic 
vector fields $\displaystyle z_i \frac{\partial}{\partial z_i}$, $i \,\in\, \{1,\, \cdots ,\, n\}$, descend 
from ${\mathbb C}^n \setminus \{0 \}$ to the quotient $M_D$ as commuting holomorphic vector fields $X_i$ on 
$M_D$ which are linearly independent outside the divisor $S$ given by the projection of the coordinate 
axis. The holomorphic geometric structure $\Psi$ given by the juxtaposition $(\nabla_0,\, X_i)$ of the 
complex affine structure and the family of commuting holomorphic vector fields $X_i$, $i \,\in\, \{1,\, 
\cdots ,\, n\}$, is a rigid holomorphic geometric structure of affine type on $M$ \cite{DG, Gr} which is 
locally homogeneous on $M\setminus S$. Indeed, the complex abelian Lie group $A$ generated by the 
holomorphic vector fields $X_i$ acts transitively on $M \setminus S$ preserving $\Psi$: $A$ coincides with 
the connected component of the identity element in the automorphism group ${\rm Aut}(M,\, \Psi)$. Moreover, $\Psi$ is 
not locally homogeneous on entire $M$ since the holomorphic section $X_1 \wedge \ldots \wedge X_n$ of the 
canonical line bundle $\bigwedge^n {T^*}M_D$ (which is invariant by all local automorphisms of $\Psi$) 
vanishes on $S$. Therefore, $M_D$ admits non-locally homogeneous rigid holomorphic geometric structures of 
affine type and admit nontrivial holomorphic tensors with non-empty vanishing locus. On the other hand diagonal Hopf manifolds are known to admit Vaisman metrics 
\cite{Bel}. Therefore, Theorem \ref{Vaisman} and Theorem \ref{Bochner} do not hold for all Vaisman 
manifolds.
\end{remark} 

\section{Principal torus bundles over K\"ahler Calabi-Yau manifolds}\label{forms}

Recall that compact complex simply connected threefolds with trivial canonical bundle do not admit {\it 
rigid} holomorphic geometric structures \cite[Theorem D]{BCDG}. The following result exhibits (non-closed) 
holomorphic one-forms on some compact simply connected threefolds with trivial canonical bundle which are 
holomorphic elliptic bundles over $K3$ surfaces. It is inspired by a construction of Brunella in \cite[page 
648]{Br}.

\begin{theorem}\label{simply connected with forms}
There exists a simply connected principal elliptic bundle $M$ over a projective $K3$ surface $B$ admitting 
a (non closed) nonsingular holomorphic one-form $\omega$.
\end{theorem} 

\begin{proof}
Let us first construct a projective $K3$-surface $B$ bearing a holomorphic two-form $\Omega$ such that the
group of periods of $\Omega$ is a lattice $\Lambda$ in $\mathbb C$.

One way to obtain such a pair $(B, \,\Omega)$ is to consider a Kummer surface. Indeed, let $\Gamma$ be the 
standard lattice in $\mathbb{C}^2$ generated by $(1,\,0),\, (\sqrt{-1},\,0),\, (0,\,1),\, (0,\,\sqrt{-1})$. 
Consider the quotient of the abelian variety $Y\,=\,\mathbb {C}^2 / \Gamma$ by the holomorphic involution 
$i\,:\, \mathbb {C}^2 / \Gamma\, \longrightarrow\, \mathbb {C}^2 / \Gamma$ defined by $(x, \,y) 
\,\longmapsto\, (-x,\, -y)$.

The above involution $i$ admits 16 fixed points which are the order 2 points in $Y$. Let $B$ be the 
blow-up of $Y/i $ at these 16 points. The smooth complex projective surface $B$ is a K3 surface (in 
particular, it is simply connected). A holomorphic 2-form $\Omega$ on B is given by the 
pull-back of the translation volume form $dz_1\wedge dz_2$ on $Y$, which is also invariant under the 
involution $i$.

The homology group $H_2(B,\, \mathbb{Z})$ is generated by the 16 copies of ${\mathbb C}P^1$ (the 
exceptional divisors for the above mentioned blow-up) and by the images in $H_2(B,\, \mathbb{Z})$ of the 6 
standard generators of $H_2(Y, \,\mathbb{Z}).$ The integral of $\Omega$ on each of the exceptional divisors is zero 
because the form $\Omega$ is pulled back from $Y/i$ and it hence vanishes on the exceptional divisors. On the 
generators of $H_2(Y,\, \mathbb{Z})$ the integral of $\Omega$ is one of the numbers $0,\, 1/4,\, -1/4,\, 
\sqrt{-1}/4,\, -\sqrt{-1}/4.$ Consequently, the group of periods of $\Omega$ form a lattice $\Lambda$ in 
$\mathbb C$.

Now consider the above pair $(B,\, \Omega)$, with $B$ being the projective K3 surface and $\Omega$ the 
(nonsingular) holomorphic two-form whose group of periods is a lattice $\Lambda$ in $\mathbb C$. The 
corresponding elliptic curve $\mathbb{C}/\Lambda$ will be denoted by $T$.

Let us choose an open cover $\{U_i\}_{i \in I}$ of $B$, by open subsets in the analytic topology, such that 
all $U_i$ and all connected components of $U_i \cap U_j$ are contractible. Then on each open subset $U_i$ 
there exists a holomorphic one-form $\omega_i$ such that $\Omega_i\,=\,d \omega_i$. On nontrivial 
intersections $U_i \cap U_j$, we get that $\omega_i-\omega_j\,=\,dF_{ij}$, with $F_{ij}$ a holomorphic 
function defined on $U_i \cap U_j$.

On intersections $U_i \cap U_j \cap U_k$, the sum $F_{ij} +F_{jk}+F_{ki}$ is a locally constant $2$-cocycle 
that represents the cohomology class in $H^2(B,\, { \mathbb C})$ of the closed two-form $\Omega$. 
Therefore, we can choose the local one-forms $\omega_{i}$ and the associated holomorphic functions $F_{ij}$ 
such that --- for every triple $(i,\, j,\, k)$ --- the value of the $2$-cocycle $F_{ij}+F_{jk}+F_{ki}$ 
lies in the lattice $\Lambda \,\subset\, \mathbb C$ formed by the periods of $\Omega$.

Consider the local holomorphic function $\{F_{ij}\}$ as a $1$-cocycle taking values in the translation 
group of $T$ and form the associated holomorphic principal elliptic bundle $M \, \longrightarrow\, B$ with 
typical fiber $T$ associated to this one-cocycle. The local forms on $U_i \times T$ which are 
$\pi_1^*(\omega_i )+\pi_2^*dz $, where $\pi_1$ (respectively, $\pi_2$) is the projection on the first 
(respectively, second) factor and $d z$ is a translation invariant holomorphic one-form on the elliptic 
curve $T$, glue compatibly to produce a globally defined holomorphic one-form $\omega$ on $M$. By 
construction, the differential $d\omega$ projects on $B$ to $\Omega$ and does not vanish. Hence, the 
holomorphic one-form $\omega$ is not closed; more precisely, its differential $d \omega$ is the pull-back 
on $M$ of the form $\Omega$. Moreover, $\omega \wedge d \omega$ does not vanish at any point in $M$ and 
provides a holomorphic section trivializing the canonical line bundle $K_M$.

Observe that the total space $M$ of the above holomorphic principal bundle
with fiber type $T$ is not K\"ahler because the holomorphic one-form $\omega$ on $M$ is not closed.

Moreover, the above manifold $M$ is simply connected, which is proved in \cite[Theorem 12.3 and Remark 
12.4]{Ho}. More precisely, employing the notation of \cite{Ho} observe that, by construction, the 
topological bundle $M$ is characterized by a characteristic class $c\,=\, a_1 \otimes \lambda_1 + a_2 
\otimes \lambda_2 \,\in\, H^2(M,\, \mathbb Z) \otimes \Lambda$, given by $\lambda_1\,=\,1/4, \lambda_2\,=\, 
\sqrt{-1}/4$ and the evaluation of $a_1$ equals $1$ on the pull-back of the class of the elliptic curve 
$E_1$ in $Y$ defined by the quotient of $\mathbb R (1,0)\oplus \mathbb R (1,0)$, equals $-1$ on the 
pull-back of the class of the elliptic curve $E_2$ in $Y$ defined by the quotient of $\mathbb R 
(0,\,1)\oplus \mathbb R (0,\,1)$ and $a_1$ vanishes on the complement in $H^2(M,\, \mathbb Z)$ of the 
direct summand generated by the pull-back of the classes of $E_1$ and $E_2$. The element $a_2$ is defined 
in a similar way using the pull-back of the classes of the elliptic curves generated by $E_3$ and $E_4$ 
which are the images in $Y$ of $\mathbb R (1,\,0)\oplus \mathbb R (0,\,1)$ and $\mathbb R (0,\,1)\oplus 
\mathbb R (1,\,0)$, respectively. Notice that in our situation the obstruction $\Delta$ in 
\cite[Proposition 6.6 and Remark 12.4]{Ho} vanishes, by construction, and since $a_1,\, a_2$ form the basis 
of a direct summand in $H^2(M,\, \mathbb Z)$ the manifold $M$ is simply connected by \cite[Theorem 
12.3]{Ho}.
\end{proof}

\begin{remark}
The manifold $M$ in Theorem \ref{simply connected with forms} is not K\"ahler because it bears a non-closed 
holomorphic one-form $\omega$. Its algebraic dimension therefore is $2$ (the projection of $M$ on $B$ 
coincides with the algebraic reduction of $M$). Notice that it comes from the proof that $\omega \wedge d 
\omega$ provides a holomorphic trivialization of the canonical line bundle of $M$. In particular, $\omega$ 
is not integrable (its kernel is not integrable and therefore it does not define a foliation on $M$). 
Recall that \cite[Proposition 1]{Br} proves that on compact complex manifolds $M$ of algebraic dimension 
$\dim M -1$, all integrable holomorphic one-forms are closed. Also recall that \cite[Theorem 1.2]{BD2} 
proves that on holomorphic principal torus bundles over compact K\"ahler Calabi-Yau manifolds all 
holomorphic geometric structures of affine type are actually locally homogeneous. This implies that those, 
among these manifolds, that admit a rigid holomorphic geometric structure of affine type must have infinite 
fundamental group \cite[Corollary 1.3]{BD2}. The manifolds in Theorem \ref{simply connected with forms} 
being simply connected do not admit any rigid holomorphic geometric structure of affine type. Moreover, the 
holomorphic one-form $\omega$ is locally homogeneous.
\end{remark}

\section*{Acknowledgments}

We thank the referee for helpful comments.
The authors would like to thank Valentino Tosatti for kindly answering their questions and pointing out 
references \cite{Boz, Ho, Fe}. We also thank Nicolina Istrati for her kind explanations about Vaisman 
Calabi-Yau manifolds and for noting a mistake in a preliminary version. The first-named author is partially 
supported by a J. C. Bose Fellowship (JBR/2023/000003). The second-named author is supported by ANR project 
IsoMoDyn ANR-25-CE40-1360-03.

\section*{Declarations}

Declaration of interests. The authors do not work for, advise, own shares in, or receive
funds from any organisation that could benefit from this article, and have declared no
affiliation other than their research organisations.

Data availability statement. No data were generated or used.


\end{document}